\documentclass[11pt]{amsart}

\usepackage{
	amsmath,
	amsthm,
	amssymb,
	tikz,
  mathrsfs,
	fancyhdr,
	enumerate,
	setspace,
	multirow,
  color
}

\usetikzlibrary{arrows,calc,matrix,shapes}

\usepackage[
	letterpaper=true,
	colorlinks=true,
	linkcolor=black,
	anchorcolor=black,
	citecolor=black,
	menucolor=black,
	filecolor=black,
	urlcolor=black
]
{hyperref}

\usepackage[center,small,sc]{caption}

\usepackage{multicol}

\theoremstyle{plain}
\newtheorem{thm}{Theorem}[section]
\newtheorem{lemma}[thm]{Lemma}
\newtheorem{prop}[thm]{Proposition}
\newtheorem{cor}[thm]{Corollary}

\theoremstyle{definition}
\newtheorem{dfn}[thm]{Definition}
\newtheorem{ex}[thm]{Example}

\newtheorem{remark}[thm]{Remark}
\numberwithin{equation}{section}
\numberwithin{figure}{section}
\numberwithin{table}{section}

\newcommand{\la}{\langle}
\newcommand{\ra}{\rangle}

\newcommand{\bw}{\boldsymbol{w}}

\newcommand{\bv}{\boldsymbol v}
\newcommand{\calC}{\mathcal C}
\newcommand{\calD}{\mathcal D}

\newcommand{\field}[1]{\mathbf{#1}}

\newcommand{\ZZ}{\field{Z}}

\tikzset{tab/.style={matrix of math nodes,column sep=-.4, row sep=-.4,text height=8pt,text width=8pt,align=center}}

\begin{document}

\title [Homogeneous representations and Dyck paths]
{Homogeneous representations of Type $A$\\ KLR-algebras and Dyck paths}
\author[G. Feinberg]{Gabriel Feinberg}
\address{Department of Mathematics and Statistics,
Haverford College, Haverford, PA 19041, U.S.A. }
\email{gfeinberg@haverford.edu}
\author[K.-H. Lee]{Kyu-Hwan Lee$^{\diamond}$}
\thanks{$^{\diamond}$This work was partially supported by a grant from the Simons Foundation (\#318706).}
\address{Department of
Mathematics, University of Connecticut, Storrs, CT 06269, U.S.A.}
\email{khlee@math.uconn.edu}

\subjclass[2010]{Primary 16G99; Secondary 05E10}
\begin{abstract}
The Khovanov-Lauda-Rouquier (KLR) algebra arose out of attempts to categorify quantum groups.  Kleshchev and Ram proved a result reducing the representation theory of these algebras to the study of irreducible cuspidal representations.  In the finite type $A$, these cuspidal representations are included in the class of homogeneous representations, which are related to fully commutative elements of the corresponding Coxeter groups. In this paper, we study fully commutative elements using combinatorics of Dyck paths. Thereby  we classify and enumerate the homogeneous representations for KLR algebras of types $A$ and obtain a dimension formula for these representations from combinatorics of Dyck paths.
\end{abstract}

\maketitle

\section*{Introduction}

Introduced by Khovanov and Lauda \cite{Khov2009} and independently by Rouquier \cite{Rouq2008}, the Khovanov-Lauda-Rouquier (KLR) algebras (also known as quiver Hecke algebras) have been the focus of many recent studies.  In particular, these algebras categorify the lower (or upper) half of a quantum group.  More precisely, the Cartan datum associated with a Kac-Moody algebra $\mathfrak g$ gives rise to a KLR algebra $R$.  The category of finitely generated projective graded  modules of this algebra can be given a bialgebra structure by taking the Grothendieck group, and taking the induction and restriction functors as multiplication and co-multiplication.  To say that the KLR algebra $R$ categorifies the negative part  $U_q^{-}(\mathfrak g)$ of the quantum group, is to say that this bialgebra is isomorphic to Lusztig's integral form of $U_q^{-}(\mathfrak g)$.

In the paper \cite{Klesh2010}, Kleshchev and Ram significantly reduce the problem of describing the irreducible representations of the KLR algebras.  They defined a class of {\em cuspidal} representations for finite types, and showed that every irreducible representation appears as the head of some induction of these cuspidals, and constructed almost all cuspidal representations. Hill, Melvin, and Mondragon in \cite{Hill2012} completed the construction of cuspidals in all finite types, and re-frame them in a more unified manner.

Furthermore, Lauda and Vazirani imposed a crystal structure on the isomorphism classes of irreducible representations of a KLR algebra.  They showed in \cite{Lauda2011} that this crystal is  isomorphic to  the crystal $B(\infty)$  of the quantum group $U_q(\mathfrak g)$.  Crystals are also used by Benkart, Kang, Oh, and Park in \cite{Benk2011} to give a new approach towards the construction of irreducible representations. For more backgrounds and other developments, see \cite{Brundan2009} and \cite{Kang2011}.

In the process of constructing the cuspidal representations, Kleshchev and Ram defined  a class of representations known as {\em homogeneous representations} \cite{Klesh2008}, those that are concentrated in a single degree.   Homogeneous representations  include most of the cuspidal representations for finite types with a suitable choice of ordering on words.  Therefore it is important to completely understand these representations.
As shown in \cite{Klesh2008}, homogeneous representations can be constructed from the sets of reduced words of {\em fully commutative} elements in the corresponding Coxeter group. These elements were studied by Fan \cite{Fan1996} and Stembridge \cite{Stembridge1996, Stembridge1998}, and are closely related to Temperley-Lieb algebras \cite{Jones1987}. 

Motivated by this connection to the homogeneous representations of KLR algebras, we study, in this paper, fully commutative elements of the Coxeter groups of type $A_n$. We decompose the set of fully commutative elements into natural subsets according to the lengths of fully commutative elements, and study combinatorial properties of these subsets. 
Our main result (Theorem \ref{thm-A}) shows that the fully commutative elements of a given length $k$ can be parametrized by the Dyck paths of semi-length $n$ with the property that  (\textit{sum of peak heights}) $-$ (\textit{number of peaks}) = $k$.
The main idea of the proof is to investigate a canonical form of reduced words for fully commutative elements. 

After the parametrization is obtained, we classify and enumerate the homogeneous representations of KLR algebras of type $A$ according to the decomposition of the set of fully commutative elements (Corollaries \ref{cor-AA}). In their paper \cite{Klesh2008}, Kleshchev and Ram gave a parametrization of homogeneous representations using {\em skew shapes}. Our result uses different combinatorial objects, i.e. Dyck paths, and gives a refinement of the classification.
Furthermore, we obtain a dimension formula for some homogeneous representations using combinatorics of Dyck paths (Proposition \ref{AForm}), which is a reformulation of the Peterson-Proctor formula.

The outline of this paper is as follows. In Section \ref{first}, we fix notations, briefly review the representations of KLR algebras,  and explain the relationship between homogeneous representations and fully commutative elements of a Coxeter group. In Section \ref{ch:A}, we introduce Dyck paths and study a canonical form of reduced words of fully commutative elements and obtain main results of this paper. In Section \ref{last}, we prove a dimension formula for homogeneous representations when the corresponding Dyck paths satisfy a certain condition.

\subsection*{Acknowledgments}
Part of this research was performed	while both authors were visiting Institute for Computational and Experimental Research in Mathematics (ICERM) during the spring of 2013 for the special program ``Automorphic Forms, Combinatorial Representation Theory and Multiple Dirichlet Series".  They wish to thank the organizers and staffs.

\section{KLR Algebras and Homogeneous Representations} \label{first}

\subsection{Definitions}
To define a KLR algebra, we begin with a quiver $\Gamma$.  In this paper, we will focus mainly on quivers of Dynkin types $A_n$, but for the definition, any finite quiver with no double bonds will suffice.  Let $I$ be the set indexing the vertices of $\Gamma$,
and for indices $i\neq j$, we will say that $i$ and $j$ are neighbors if $i\rightarrow j$ or $i \leftarrow j$.  Define $Q_+ = \bigoplus_{i\in I} \mathbb{Z}_{\geq 0} \, \alpha_i$ as the non-negative lattice with basis $\{\alpha_i | i\in I\}$.
The set of all words in the alphabet $I$ is denoted by $\langle  I \rangle$, and for a fixed $\alpha = \sum_{i\in I}c_i\alpha_i \in Q_+$, let $\langle I \ra_\alpha$ be the set of words $\bw$ on the alphabet $I$ such that each $i\in I$ occurs exactly $c_i$ times in $\bw$.  We define the {\em height}   of $\alpha$ to be $\sum_{i\in I}c_i$.  We will write $\bw = [{w_1}, {w_2}, \hdots, {w_d}]$, $w_j \in I$.

Now, fix an arbitrary ground field $\mathbb F$ and choose an element $\alpha\in Q_+$.  Then the \textit{Khovanov-Lauda-Rouquier algebra} $R_\alpha$ is the associative $\mathbb F$-algebra generated by:
 \begin{itemize}
  \item  idempotents
  $\{e(\bw) ~|~ \bw\in \la I \ra_\alpha\}$,

   \item  symmetric generators $\{\psi_1, \hdots, \psi_{d-1}\}$ where $d$ is the height of  $\alpha$,
  \item  polynomial generators
  $\{y_1, \hdots, y_d\}$,
  \end{itemize}

  subject to relations
  \begin{align} \label{eqn-1}
 & e(\bw)e(\bv) = \delta_{\bw\bv}e(\bw), \quad \quad \sum_{\bw\in \la I \ra_\alpha}e(\bw) = 1;
  \\ & y_ke(\bw) = e(\bw)y_k;
   \\  \label{eqn-2}
& \psi_ke(\bw) = e(s_k\bw)\psi_k;
  \\ 
& y_ky_\ell = y_\ell y_k;
 \\ 
& y_k\psi_\ell =\psi_\ell y_k \quad \hbox{for } k\neq \ell,\ell+1;
  \\  &   (y_{k+1}\psi_k - \psi_ky_k)e(\bw) =
    \left\{ \begin{array}{l l}
     e(\bw) & \hbox{if } w_k = w_{k+1},\\
      0 & \hbox{otherwise; }\\
    \end{array}\right.
   \\    & (\psi_k y_{k+1}- y_k\psi_k)e(\bw) =
    \left\{ \begin{array}{l l}
     e(\bw) & \hbox{if } w_k = w_{k+1},\\
      0 & \hbox{otherwise; }\\
    \end{array}\right.
      \\ 
    & \psi_k^2e(\bw) =
    \left\{ \begin{array}{l l}
    0 & \hbox{if } w_k = w_{k+1},\\
      (y_k-y_{k+1})e(\bw) & \hbox{if } w_k\rightarrow w_{k+1},\\
      (y_{k+1}-y_{k})e(\bw) & \hbox{if } w_k\leftarrow w_{k+1},\\
     e(\bw) & \hbox{otherwise; }\\
    \end{array}\right.
  \\ 
  &  \psi_k\psi_\ell = \psi_\ell\psi_k \quad \hbox{for } |k-\ell|>1 ;
  \\  &  (\psi_{k+1}\psi_{k}\psi_{k+1} -\psi_{k}\psi_{k+1}\psi_{k})e(\bw) =
    \left\{ \begin{array}{l l}
       e(\bw)& \hbox{if } w_{k+2} = w_k\rightarrow w_{k+1},\\
       -e(\bw)& \hbox{if } w_{k+2} = w_k\leftarrow w_{k+1},\\
     0 & \hbox{otherwise. }\\
    \end{array}\right.
  \end{align}
Here $\delta_{\bw\bv}$ in \eqref{eqn-1} is the Kronecker delta and, in \eqref{eqn-2}, $s_k$ is the $k^\textrm{th}$ simple transposition in the symmetric group $S_d$, acting on the word $\bw$ by swapping the letters in the $k^\textrm{th}$ and $(k+1)^\textrm{st}$ positions.
If $\Gamma$ is a Dynkin-type quiver, we will say that $R_\alpha$ is a KLR algebra of that type.

We impose a $\mathbb Z$-grading on $R_\alpha$ by
\begin{align} \label{eqn-degree}
 & \deg(e(\bw))  =  0, \quad
 \deg(y_i) = 2, \\  & \deg(\psi_i e(\bw))  =   \left\{\begin{array}{rl}
 -2 & \hbox{ if } w_i= w_{i+1},\\
 1 & \hbox{ if } w_i, w_{i+1} \hbox{ are neighbors in } \Gamma ,\\
 0 & \hbox{ if } w_i, w_{i+1} \hbox{ are not neighbors in } \Gamma .\\
\end{array}\right. \label{eqn-degree-1}
\end{align}

Set $R = \bigoplus_{\alpha\in Q_+} R_\alpha$, and let $\text{Rep}(R)$ be the category of finite dimensional graded $R$-modules, and denote its Grothendieck group by $[\text{Rep}(R)]$. Then $\text{Rep}(R)$ categorifies one half of the quantum group. More precisely, let $\mathbf f$ and $'\mathbf f$ be the Lusztig's algebras defined in \cite[Section 1.2]{lusztig2011introduction} attached to the Cartan datum encoded in the quiver $\Gamma$ over the field $\mathbb Q(v)$. We put $q=v^{-1}$ and $\mathcal A=\mathbb Z[q, q^{-1}]$, and let $'\mathbf f_{\mathcal A}$ and $\mathbf f_{\mathcal A}$ be the $\mathcal A$-forms of $'\mathbf f$ and $\mathbf f$, respectively. Consider the graded duals $'\mathbf f^*$ and $\mathbf f^*$, and their $\mathcal A$-forms \['\mathbf f^*_{\mathcal A} := \{ x \in {'\mathbf f}^* : x('\mathbf f_{\mathcal A}) \subset \mathcal A\} \ \text{ and } \ \mathbf f^*_{\mathcal A} := \{ x \in \mathbf f^* : x(\mathbf f_{\mathcal A}) \subset \mathcal A\}.\]
Then Khovanov and Lauda  \cite{Khov2009} prove that 
there is an $\mathcal A$-linear (bialgebra) isomorphism $[\text{Rep}(R)]  \xrightarrow{\sim} \mathbf f^*_{\mathcal A}$.
 More details  can be found in \cite{Khov2009, Klesh2010}.

\medskip

A word $\mathbf i \in \la I \ra_\alpha$ is naturally considered as an element of $'\mathbf f_{\mathcal A}^*$ to be dual to the corresponding monomial in $'\mathbf f_{\mathcal A}$.
Let $M$ be a finite dimensional graded $R_\alpha$-module. Define the $q$-character of $M$ by
\[ \text{ch}_q \, M := \sum_{\mathbf i \in \la I \ra_\alpha} (\dim_q M_{\mathbf i} )\, \mathbf i \in {'\mathbf f}^*_{\mathcal A} ,\] where $M_{\mathbf i} = e(\mathbf i) M$ and $\dim_q V:= \sum_{n \in \mathbb Z} (\dim V_n) \, q^n \in \mathcal A$ for $V=\oplus_{n \in \mathbb Z} V_n$.
A non-empty word $\mathbf i$ is called {\em Lyndon} if it is lexicographically smaller than all its proper right factors with respect to a fixed ordering on $I$.
For $x \in {'\mathbf f}^*$ we denote by $\max(x)$ the largest word appearing in $x$. A word $\mathbf i \in \langle I \rangle$ is called {\em good} if there is $x \in \mathbf f^*$ such that $\mathbf i = \max(x)$. Given a module $L \in \text{Rep}(R_\alpha)$, we say that $\mathbf i \in \langle I \rangle$ is the {\em highest weight} of $L$ if $\mathbf i = \max(\text{ch}_q \, L)$. An irreducible  module 
$L \in \text{Rep}(R_\alpha)$ is called {\em cuspidal} if its highest weight is a good Lyndon word.

The following theorem explains the importance of cuspidal representations as building blocks for all irreducible representations of $R_\alpha$. 
\begin{thm}[\cite{Klesh2010}; \cite{Hill2012}, 4.1.1] Assume that $\Gamma$ is of finite Dynkin type. Then any irreducible graded $R_\alpha$-module  for $\alpha \in Q_+$ is given by an  irreducible head of a standard representation induced from cuspidal representations up to isomorphism and degree shift.
\end{thm}


\subsection{Homogeneous representations}
  We define a \textit{homogeneous representation} of a KLR algebra to be an irreducible, graded representation fixed in a single degree (with respect to the $\mathbb Z$-grading described in  \eqref{eqn-degree} and \eqref{eqn-degree-1}).   
Homogeneous representations form an important class of irreducible modules since most of the cuspidal representations are  homogeneous  with a suitable choice of ordering on $\la I \ra$ (\cite{Klesh2010, Hill2012}). After introducing some terminology, we will  describe these representations in a combinatorial way. We continue to assume that $\Gamma$ is a simply-laced quiver.

  Fix an $\alpha\in Q_+$ and let $d$ be the height of $\alpha$. For any word $\bw\in \la I \ra_\alpha$, we say that the simple transposition $s_r\in S_d$ is an \textit{admissible transposition for $\bw$} if the letters $w_r$ and $w_{r+1}$ are neither equal nor neighbors in the quiver $\Gamma$.  Following Kleshchev and Ram \cite{Klesh2008}, we define the \textit{weight graph} $G_\alpha$ with vertices given by $\la I \ra_\alpha$.  Two words $\bw$, $\bv \in \la I \ra_\alpha$ are connected by an edge if there is an admissible transposition $s_r$ such that $s_r\bw = \bv$.

We say that a connected component $C$ of the weight graph $G_\alpha$ is \emph{homogeneous} if the following property holds for every $\bw\in C$:
  \begin{eqnarray} \label{homog}
 &   &  \hbox{If } w_r  =  w_s  \hbox{ for some } 1\leq r<s \leq d \hbox{, then there exist } t,u\\\nonumber
    & & \hspace*{1.8 cm} \hbox{ with } r<t<u<s \hbox{ such that }  w_r \hbox{ is neighbors with both } w_t \hbox{ and } w_u.
  \end{eqnarray}
A word satisfying condition~\eqref{homog} will be called a \textit{homogeneous word}.

 A main theorem of \cite{Klesh2008} shows that the homogeneous components of $G_\alpha$ exactly parameterize the homogeneous representations of the KLR algebra $R_\alpha$:

 \begin{thm}[\cite{Klesh2008}, Theorem 3.4] \label{com} Let $C$ be a homogeneous component of the weight graph $G_\alpha$.  Define an $\mathbb F$-vector space $S(C)$ with basis $\{v_{\bw}~|~\bw\in C \}$ labeled by the vertices in $C$.  Then we have an $R_\alpha$-action on $S(C)$ given by
  \begin{eqnarray*}
  	e(\bw')v_{\bw} & = & \delta_{\bw,\bw'}v_{\bw} \quad (\bw'\in \la I \ra_\alpha, \bw\in C), \\
  	y_rv_{\bw} &=& 0 \quad (1\leq r \leq d, \bw\in C), \\
  	\psi_rv_{\bw} & = &
  	  \left\{\begin{array}{ll}
  	  	v_{s_r\bw} & \textrm{ if } s_r\bw\in C \\
  	  	0 & \textrm{otherwise} 
  	  \end{array}\right.   \quad (1\leq r \leq d-1, \bw\in C), 
  \end{eqnarray*}
  which gives $S(C)$ the structure of a homogeneous, irreducible $R_\alpha$-module.  Further $S(C)\ncong S(C')$ if $C\neq C'$, and this construction gives all of the irreducible homogeneous modules, up to isomorphism.
 \end{thm}

As a result, the task of identifying homogeneous representations of a KLR algebra is reduced to identifying homogeneous components in a weight graph.  This is simplified further by the following lemma:

  \begin{lemma}[\cite{Klesh2008}, Lemma 3.3] A connected component $C$ of the weight graph $G_\alpha$ is homogeneous if and only if an element $\bw\in C$ satisfies the condition~\eqref{homog}.
  \end{lemma}

 Recall that we call a word satisfying condition~\eqref{homog} a {homogeneous word}.  The homogeneous words have other combinatorial characterizations, which we explore in the next subsection.


\subsection{Fully commutative elements of Coxeter groups}

Since the homogeneity of $\bw \in \la I \ra$ does not depend on the  orientation of a quiver,  it is enough to consider Dynkin diagrams and the corresponding Coxeter groups.
Given a simply laced Dynkin diagram,  the corresponding Coxeter group will be denoted by $W$ and the generators by $s_i$, $i \in I$. A reduced expression $s_{i_1} \cdots s_{i_r}$ will be identified with the word $[i_1,  \dots, i_r]$ in  $ \la I \ra$. The identity element will be identified with the empty word $[~]$.

An element $\bw \in W$ is said to be {\em fully commutative} if any reduced word for $\bw$ can be obtained from any other  by  interchanges of adjacent commuting generators, or equivalently if no reduced word for $\bw$ has $[i, i', i]$ as a subword where $i$ and $i'$ are neighbors in the Dynkin diagram. Now we have the following lemma, which was first observed by  Kleshchev and Ram.

\begin{lemma} \cite{Klesh2008} \label{lem-equiv}
\hfill
\begin{enumerate}
\item A homogeneous component of the weight graph  $G_\alpha$ contains as its vertices exactly the set of reduced expressions for a fully commutative element in $W$. 

\item The set of homogeneous components is in bijection with the set of fully commutative elements in $W$.

\end{enumerate}

\end{lemma}

Stembridge \cite{Stembridge1996} classified all of the Coxeter groups that have finitely many fully commutative elements, completing the work of Fan \cite{Fan1996}, who had done this for the simply-laced types.
Fan and Stembridge also enumerated the set of fully commutative elements.  In particular, they showed that  the number of fully commutative elements in the Coxeter group of type $A_n$ is $C_{n+1}$, where $C_n$ be the $n^\textrm{th}$ Catalan number, i.e. $C_n = \frac{1}{n+1}{2n\choose n}$. This fact has an immediate implication on homogeneous representations by Lemma \ref{lem-equiv}.
\begin{cor}
 A KLR algebra $R = \bigoplus_{\alpha\in Q_+}R_\alpha$ of type $A_n$ has $C_{n+1}$ irreducible homogeneous representations.
\end{cor}

In \cite{Klesh2008}, Kleshchev and Ram parameterized homogeneous representations using skew shapes. In this paper, we will decompose the set of fully commutative elements to give a finer enumeration of homogeneous representations in type $A_n$. 
More precisely, in the next section, our main result is a fine bijection between the family of irreducible homogeneous representations and the set of {\em Dyck paths} in accordance with the decomposition of the set of fully commutative elements.  This bijection can be used to quickly enumerate the fully commutative elements of a given length and the attached homogeneous representations.


\section{Homogeneous Representations of Type $A_n$ KLR Algebras} \label{ch:A}
 In this section, we describe all of the homogeneous representations of a KLR algebra of type $A_n$, associated with a quiver whose underlying graph is
\[
 \begin{tikzpicture}[x=2cm, scale=.5]
    \foreach \x in {0,1, 3}
    \draw[xshift=\x,thick, fill=black] (\x,0) circle (1.5 mm);
    \draw (0,0)--(1,0); 
\draw (2,0)--(3,0);
    \draw[dotted,thick] (1,0) -- (2,0);
    \draw (0,-.3) node[below]{\footnotesize1};
 \draw (1,-.3) node[below]{\footnotesize2};
    \draw (3,-.3) node[below]{\footnotesize $n$};
  \end{tikzpicture} 
  \]
  We begin by introducing the main combinatorial tool for our study.

\subsection{Dyck paths}
As in \cite{Deutsch1999}, we define a {\em Dyck path} as a lattice path in the first quadrant consisting of steps $\langle1,1\rangle$ (north-east) and $\langle1,-1\rangle$ (south-east), beginning at the origin and ending at the point $(2n,0)$.  We refer to $n$ as the semi-length of the path.  By a \textit{peak} we shall mean a rise $\langle 1,1 \rangle$ followed by a fall $\langle 1,-1 \rangle$, while a \textit{valley} is a fall, followed by a rise.

\begin{figure}[h]
\[
\begin{tikzpicture}[x=1.05cm,y=1.05cm, scale=0.7]
\foreach \x in {0,2,4,6,8,10}
\draw[shift={(\x-1,0)},color=black, ] (0pt,2pt) -- (0pt,-2pt) node[below] {\footnotesize $(\x,0)$};
\draw[line width=1 pt, ->] (-1,0)--(10,0);
\draw[orange, line width=1.5pt] (-1,0)--(0,1)--(1,2)--(2,1)--(3,2)--(4,3)--(6,1)--(7,0)--(8,1)--(9,0);
\foreach \y in {1,2,3}
	\draw[dashed](-1,\y)--(10,\y);
\end{tikzpicture}
\]
\caption{An example of a Dyck path of semilength 5.}\label{DyckEx}
\end{figure}

  Denote by $\mathcal D_{n,k}$ the set of all Dyck paths of semi-length $n$ with the property that  (\textit{sum of peak heights}) $-$ (\textit{number of peaks}) = $k$.  For the example path shown in figure~\ref{DyckEx} we have $k=(2+3+1) - 3 = 3$.

  Let $T(n,k)$ be the cardinality of the set $\calD_{n,k}$, as defined in \cite{Tnk}.  It is known that $T(n,k)=0$ when $k>1+\lfloor\frac{n^2}{4}\rfloor$. It is convenient to display the sequence of non-zero values as an array with the entry $T(n,k)$ in the $n^\textrm{th}$ row from the top (starting with $n=0$) and the $k^\textrm{th}$ column (beginning with $k=0$).  The top of the array is shown below.

  \begin{equation} \label{array}
\begin{array}{lllllllllll}
1 & \\
1 &  \\
1 & 1\\
1 & 2 & 2\\
1 & 3 & 5  & 4  & 1 \\
1 & 4 & 9 & 12 & 10 & 4 & 2\\
1 & 5 & 14 & 25 & 31 & 26 & 16 & 9 & 4 & 1 \end{array}
\end{equation}

It is well known that the number of Dyck paths of semi-length $n$ is equal to the $n^\textrm{th}$ Catalan number, $ C_n= \frac{1}{n+1}{2n \choose n}$, so we have 
\begin{equation} \label{tnk}
C_n = \sum_k T(n,k), 
\end{equation}
i.e. the sum of entries on the $n^\textrm{th}$ row is equal to $C_n$.

\medskip

Now we consider fully commutative elements in the  Coxeter group $W$ of type $A_n$ and state the main result in this section. Recall that the length of  an element of $W$ is defined with respect to the generators of $W$.

\begin{thm} \label{thm-A}
Let $\calC_{n,k}$ be the set of fully commutative elements of length $k$ in $W$ for $k\ge 0$. 
Then there is a natural bijection $\Phi: \calC_{n,k} \to \calD_{n+1,k}$. In particular, we have, for $k \ge 0$,
\[ | \calC_{n,k} | = T(n+1,k) .\]
\end{thm}

By  Lemma \ref{lem-equiv}, we will identify $\calC_{n,k}$ with the set of homogeneous components of weight graphs $G_\alpha$ with $\alpha$ having height $k$. It follows from Theorem \ref{com} that  a homogeneous representation is completely determined by a homogeneous component of a weight graph. Thus Theorem \ref{thm-A} implies the following results regarding the homogeneous representations.

    \begin{cor}  \label{cor-AA} \hfill
\begin{enumerate}
\item      There exists a bijection between the irreducible homogeneous representations of a KLR algebra of type $A_n$ and the Dyck paths of semi-length $n+1$.  Further, if such a representation is given by a homogeneous component of words with length $k$, the corresponding Dyck path has (sum of peak heights) $-$ (number of peaks) = $k$. 
    \item 
The total number of homogeneous representations of a KLR algebra of type $A_n$, which are given by homogeneous words of length $k$, is $T(n+1,k)$.
\end{enumerate}
\end{cor}

We will prove Theorem \ref{thm-A} in Section \ref{bij}, after we construct canonical words for fully commutative elements in the next subsection. 

\subsection{Canonical reduced words}
\label{sec:Canonical A}
 We define the decreasing segments
 \[ T_i^j = \left\{
     \begin{array}{lr}
       ~[j, j-1, \hdots, i+1, i] & \hbox{ for } i \le j, \\
             ~[~] & \hbox{ for } i>j .\\
     \end{array}
   \right. \]
The word $T_i^j$ will also be considered as the element $s_j s_{j-1} \cdots s_{i+1}s_i \in W (\cong S_{n+1})$. In particular, the product $T_i^jT_{i'}^{j'}$ given by concatenation is well defined.

These segments will be fundamental, so we record some facts here that we will use freely.

\begin{lemma}
Let $T_i^j$ be a segment, as defined above.  Then we have, for $i, i' , j, j' \in I$,
\begin{enumerate}
 \item $T_i^j$ is a homogeneous word;
 \item If $i-1=j'\geq i'$ then $T_i^jT_{i'}^{j'} = T_{i'}^{j}$;
 \item If $j'<i-1$ then $T_i^jT_{i'}^{j'} = T_{i'}^{j'}T_i^j$.
\end{enumerate}
\end{lemma}

\begin{proof}
  These statements follow directly from the definitions.
\end{proof}

We can use these segments to obtain a canonical form for the elements in the Coxeter group $W$ of type $A_n$:
\begin{lemma} \label{BS form}
 Every element  in $W$ of type $A_n$ can be written in the form
 \begin{equation} \label{fform}
 T^1_{i_1}T^2_{i_2}\cdots T^n_{i_n}
 \end{equation}
  where $1\leq i_j \leq j+1$ for all $1 \leq j \leq n$.
\end{lemma}

\begin{remark}
As a check, notice that there are $(n+1)!$ choices for the $i_j$'s in this form, and hence $(n+1)!$ elements in $W$. 
The above lemma is standard.
One can find a proof in  Lemma 3.2 of \cite{Bokut2001}, which  uses a Gr{\"o}bner--Shirshov basis.
 \end{remark}

 Using the canonical form \eqref{fform}, we can describe canonical representatives of  homogeneous components  or  fully commutative elements of $W$ in a coherent way.

\begin{prop} \label{can form}
  Every homogeneous component of a weight graph contains a unique word of the form
  \begin{equation} \label{eqn-fc}
T_{i_1}^{m_1}T_{i_2}^{m_2}\cdots T_{i_\ell}^{m_\ell}\end{equation}
  where $i_j\leq m_j$ for each $j$, $1\leq i_1< i_2 < \cdots < i_\ell\leq n$ and  $ m_1 < m_2 <\cdots < m_\ell$. Equivalently, a fully commutative element of $W$ can be uniquely written in the form \eqref{eqn-fc}.
\end{prop}

\begin{proof}
Clearly, every homogeneous component has a unique word of the form \eqref{fform}. After omitting, if any, segments of the form $T^j_{j+1}$, we obtain  $\bw=T_{i_1}^{m_1}T_{i_2}^{m_2}\cdots T_{i_\ell}^{m_\ell}$ with $i_j\leq m_j$ for each $j$ and  $ m_1 < m_2 <\cdots < m_\ell$. We need only to prove $1\leq i_1< i_2 < \cdots < i_\ell\leq n$.  For the sake of contradiction, assume that $i_r \geq i_s$ for some $r<s$. Without loss of generality, suppose that $i_1\geq i_2$.  Then $\bw$ has as a subword $[m_1, \hdots, i_1,  m_2, \hdots, i_1, \hdots, i_2]$.  But this subword has two occurrences of the letter $i_1$ separated by only one neighbor $i_1+1$, therefore violating the homogeneity assumption. The equivalence of the second assertion follows from Lemma \ref{lem-equiv}.
\end{proof}

\subsection{A bijection---proof of Theorem \ref{thm-A}} \label{bij} 
Recall that  $\calC_{n,k}$ is the set of fully commutative elements of length $k$ in $W$ for $k\ge 0$.  By Lemma \ref{lem-equiv}, we will also consider $\calC_{n,k}$ as the set of homogeneous components from all weight graphs $G_\alpha$ with $\alpha$ having height $k$.
We need to  establish a bijection  $\Phi$: $\calC_{n,k}\to \calD_{n+1,k}$ to prove Theorem \ref{thm-A}. We first 
construct a lattice as shown in Figure~\ref{lattice}, ranging (horizontally) from $(0,0)$ to $(2n+2,0)$. Notice that each square block corresponds to $T_i^j$ for some $i\le j$, and a Dyck path can have peaks at squares $T_i^j$ or at bottom triangles. Now suppose that we have a homogeneous component  $C \in \calC_{n,k}$.  By  Proposition \ref{can form}, we can choose a canonical representative $\bw =T_{i_1}^{m_1}T_{i_2}^{m_2}\cdots T_{i_\ell}^{m_\ell} $  with $i_j\leq m_j$ for each $j$, where $i_1< i_2 < \cdots < i_\ell$ and  $ m_1 < m_2 <\cdots < m_\ell$.  

\begin{dfn}
Suppose that $C \in \calC_{n,k}$ and $\bw=T_{i_1}^{m_1}T_{i_2}^{m_2}\cdots T_{i_\ell}^{m_\ell} $ are as above. Then the Dyck path $\Phi(C)$ is defined to be the  path with peaks only  at the square blocks containing $T_{i_j}^{m_j}$ ($j=1, 2, \dots, \ell$) and possibly, at bottom triangles.
\end{dfn}

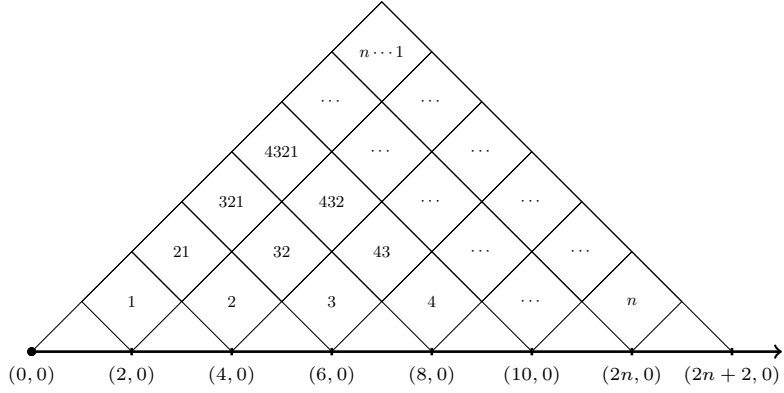
\begin{figure}
\[
\begin{tikzpicture}[x=.95cm,y=.95cm, scale=0.7]
    \foreach \x in {0,2,4,6,8,10}
\draw[shift={(\x-1,0)},color=black,line width=1pt ] (0pt,2pt) -- (0pt,-2pt) node[below] {\tiny $(\x,0)$};
\draw[shift={(12-1,0)},color=black,line width=1pt ] (0pt,2pt) -- (0pt,-2pt) node[below] {\tiny $(2n,0)$};
\draw[shift={(14-1,0)},color=black,line width=1pt ] (0pt,2pt) -- (0pt,-2pt) node[below] {\tiny $(2n+2,0)$};
\draw[line width=1pt, ->] (-1,0)--(14,0);
\draw(-1,0)--(6,7)--(13,0);
  \draw[nodes={draw, diamond, line width=.4pt, scale=0.7}, minimum height=1.9cm, minimum width=1.9cm,font=\footnotesize]
  (1,1) node{1}
   (3,1) node{2}
   (5,1) node{3}
   (7,1) node{4}
   (9,1) node{$\cdots$}
   (11,1) node{$n$}
   (2,2) node{21}
   (4,2) node{32}
   (6,2) node{43}
   (8,2) node{$\cdots$}
   (10,2) node{$\cdots$}
   (3,3) node{321}
   (5,3) node{432}
   (7,3) node{$\cdots$}
   (9,3) node{$\cdots$}
   (4,4) node{4321}
   (6,4) node{$\cdots$}
   (8,4) node{$\cdots$}
   (5,5) node{$\cdots$}
   (7,5) node{$\cdots$}
   (6,6) node{$n \cdots 1$};
   
   \foreach \x in {0,2,...,12}
   \draw[fill= black] (-1,0) circle (2pt);
\end{tikzpicture}
\]
\caption{The triangular lattice for tracing Dyck paths} \label{lattice}
\end{figure}

Before we check that the map $\Phi$ is well-defined, i.e. $\Phi(C) \in \calD_{n+1,k}$, we consider an example to see how the definition works. 

\begin{ex} \label{ex-1}
Suppose the quiver $\Gamma$ is of type $A_4$, and the homogeneous component $C$ is

  
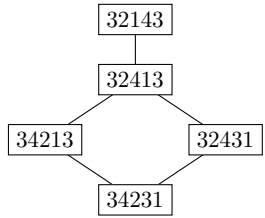
\[
  \begin{tikzpicture}[x=1.5cm, y=1cm, scale=0.8]
  \draw[nodes={draw,rectangle, fill=white ,fill opacity=1, scale=0.8}]
        (0,4) node{32143}--(0,3) node{32413}
     --(-1,2)node{34213}
     --(0,1)node{34231}
     --(1,2)node{32431}
     --(0,3);
  \end{tikzpicture}
  \]

  Then the canonical representative of this component is $\bw = [3,2,1,4,3] = T_1^3T_3^4$, and the Dyck path $\Phi(C)$ is given by:
 
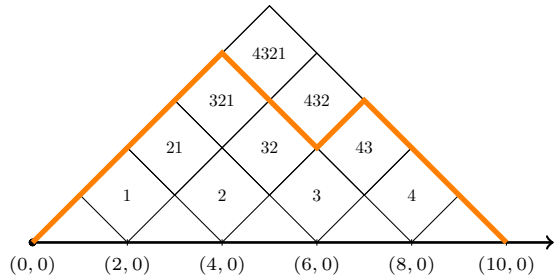
\[     \begin{tikzpicture}[x=1.05cm,y=1.05cm, scale=0.6]
    \foreach \x in {0,2,4,6,8,10}
\draw[shift={(\x-1,0)},color=black, ] (0pt,2pt) -- (0pt,-2pt) node[below] {\tiny $(\x,0)$};
\draw[line width=1pt, ->] (-1,0)--(10,0);
\draw(-1,0)--(4,5)--(9,0);
    \draw[nodes={draw, diamond, line width=.4pt, scale=0.6}, minimum height=2.1cm, minimum width=2.1cm]
  (1,1) node{1}
   (3,1) node{2}
   (5,1) node{3}
   (7,1) node{4}
   (2,2) node{21}
   (4,2) node{32}
   (6,2) node{43}
   (3,3) node{321}
   (5,3) node{432}
   (4,4) node{4321};
   \foreach \x in {0,2,...,8}
   \draw[fill= black] (-1,0) circle (2pt);
     \draw[orange, line width=2 pt] (-1,0)--(3,4)--(5,2)--(6,3)--(9,0);
\end{tikzpicture} \]
Here the sum of peak heights of the Dyck path is $4+3=7$, while the number of peaks is $2$. Then we see that $k=7-2=5$ is equal to the length of the corresponding word $\bw=[3,2,1,4,3]$.
\end{ex}

\begin{lemma} \label{lem}
The map $\Phi$ is well-defined.
\end{lemma}

\begin{proof}
Let $ C\in \calC_{n,k}$ be a homogeneous component with canonical representative \[\bw=T_{i_1}^{m_{1}}T_{i_2}^{m_{2}}\cdots T_{i_\ell}^{m_{\ell}}.\] Since $i_1< i_2 < \cdots < i_\ell$ and  $ m_1 < m_2 <\cdots < m_\ell$, there is no redundancy among the peaks and the corresponding Dyck path $D$ is uniquely determined. Note that each segment $T_{i_j}^{m_{j}}$ contains $m_{j} - i_j +1$ letters.  Since $\bw$ has $k$ letters by assumption, we have
    \begin{eqnarray*}
      k & = & \sum_{j =1}^\ell [(m_{j} - i_j) +1] = \ell + \sum_{j =1}^\ell (m_{j} - i_j). \\
    \end{eqnarray*}
On the other hand, in the path $\Phi(C)$, each of the $\ell$ segments $T_{i_j}^{m_{j}}$ corresponds to a peak with height $(m_{j} - i_j) +2$.  We then have
    \[
      \hbox{(peak heights)} -\hbox{(\# of peaks)} = \sum_{j=1}^\ell[(m_{j} - i_j) +2] - \ell = \ell + \sum_{j =1}^\ell (m_{j} - i_j) =k.
    \]
Thus $\Phi(C) \in \calD_{n+1,k}$ as desired.
\end{proof}

\begin{dfn}
  To define the inverse, $\Psi:\calD_{n+1,k}\to\calC_{n,k}$, we simply read the words contained in the square blocks of the peaks of the Dyck path $D$ from left to right, ignoring peaks at bottom triangles. Then we   obtain $\bw =T_{i_1}^{m_1}T_{i_2}^{m_2}\cdots T_{i_\ell}^{m_\ell} $, and $\bw$ determines the corresponding homogeneous component $\Psi(D)$.
  \end{dfn}
 
 A similar argument as in Lemma \ref{lem} shows that the map $\Psi$ is well-defined. 
  
\begin{ex}  \label{ex-2}
 Suppose that we have the Dyck path $D$:
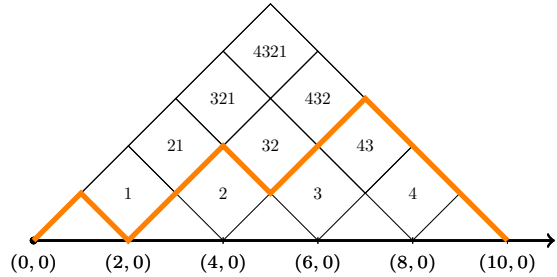
\[       \begin{tikzpicture}[x=1.05cm,y=1.05cm, scale=0.6]
    \foreach \x in {0,2,4,6,8,10}
\draw[shift={(\x-1,0)},color=black, ] (0pt,2pt) -- (0pt,-2pt) node[below] {\tiny $(\x,0)$};
\draw[line width=1pt, ->] (-1,0)--(10,0);
\draw(-1,0)--(4,5)--(9,0);
    \draw[nodes={draw, diamond, line width=.4pt, scale=0.6}, minimum height=2.1cm, minimum width=2.1cm]
  (1,1) node{1}
   (3,1) node{2}
   (5,1) node{3}
   (7,1) node{4}
   (2,2) node{21}
   (4,2) node{32}
   (6,2) node{43}
   (3,3) node{321}
   (5,3) node{432}
   (4,4) node{4321};
   \foreach \x in {0,2,...,8}
   \draw[fill= black] (-1,0) circle (2pt);
      \foreach \x in {0,2,4,6,8,10}
\draw[shift={(\x-1,0)},color=black, ] (0pt,2pt) -- (0pt,-2pt) node[below] {\tiny $(\x,0)$};
\draw[line width=1pt, ->] (-1,0)--(10,0);
     \draw[orange, line width=2 pt] (-1,0)--(0,1)--(1,0)--(2,1)--(3,2)--(4,1)--(6,3)--(9,0);
\end{tikzpicture}
\]

Reading, from left to right, the segments contained in the peaks, we see that the component $\Psi(D)$ is represented by the word $[2,4,3]  = T_2^2T_3^4$.    This is the homogeneous component
  
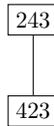
\[
  \begin{tikzpicture}[x=1.5cm, y=1cm, scale=0.8]
  \draw[nodes={draw,rectangle, fill=white ,fill opacity=1, scale=0.8}]
        (1,0) node[above]{243}--(1,-1) node[below]{423};
  \end{tikzpicture}
  \]
Note that the sum of peak heights is $1+2+3=6$ and the number of peaks is $3$. The value of $k=6-3=3$ equals the length of $\bw=[2,4,3]$. 
\end{ex}
 
 Now we complete the proof that the map $\Phi$ is a bijection with inverse $\Psi$.
 It is clear from the construction that the blocks in the lattice that form the peaks of $\Phi(C)$ contain the words, respectively, $T_{i_1}^{m_{1}}, T_{i_2}^{m_{2}}, \hdots,  T_{i_j}^{m_{\ell}}.$  Thus, we see that $\Psi(\Phi(C))= C$.
Conversely, suppose that $D\in\calD$ is a Dyck path that has been superimposed on the triangular lattice, and assume that $D$ has $\ell$  peaks ($\ell>0$) that correspond to square blocks. Reading the words occurring at each of these peaks, we obtain $T_{i_1}^{m_{1}}, T_{i_2}^{m_{2}}, \hdots,  T_{i_\ell}^{m_{\ell}}$.  We notice that, by construction, $i_1 < i _2 <\cdots < i_\ell$ and $m_{1} < m_{2} < \cdots < m_{\ell}$. By Proposition \ref{can form}, the word $T_{i_1}^{m_{1}}T_{i_2}^{m_{2}}\cdots T_{i_\ell}^{m_{\ell}}$ represents a homogeneous component.
We also see that $\Phi(\Psi(D)) = D$, and so $\Psi$ is a two-sided inverse of $\Phi$, proving the bijection. This completes the proof of Theorem \ref{thm-A}.

\section{Dimensions of Homogeneous Representations} \label{last}
In \cite{Klesh2008}, Kleshchev and Ram explain how each fully commutative element $\bw$ of the Weyl group $A_n$ can be associated to an abacus diagram, which gives rise to a skew tableau $\mathcal{ \lambda}$.  Further, if $\bw$ is a dominant minuscule element, the Peterson-Proctor hook formula applied to this tableau will count the number of reduced expressions for the fully commutative element, and thus count the dimension of the corresponding $R_\alpha$-module.

In this section, we will adopt our parameterization of the homogeneous modules and obtain a dimension formula only using combinatorics of Dyck paths.   
We begin by extending the ascents on the Dyck path to connect peaks of the Dyck path with the corresponding points on the $x$-axis, and highlighting any block that appears on one of these extended ascents. For example, we have

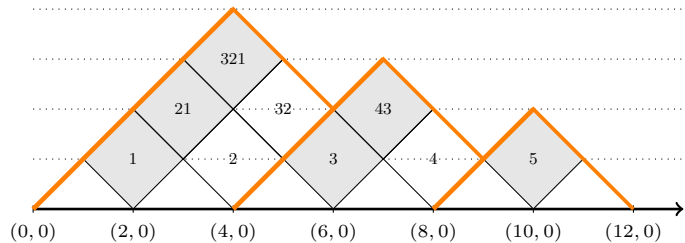
\begin{figure}[h]
\[
\begin{tikzpicture}[x=.95cm,y=.95cm, scale=0.7]
\foreach \x in {0,2,...,12}
\draw[shift={(\x-1,0)},color=black, ] (0pt,2pt) -- (0pt,-2pt) node[below] {\tiny $(\x,0)$};
\draw[line width=1 pt, ->] (-1,0)--(12,0);
\foreach \y in {1,..., 4}
	\draw[dotted](-1,\y)--(12,\y);
\draw[dashed] 
     (3,0)--(6,3)
     (7,0)--(9,2);
     \draw[nodes={draw, diamond, line width=.4pt, scale=0.7,fill=gray!20}, minimum height=1.9cm, minimum width=1.9cm,font=\footnotesize]
  (1,1) node{1}
   (5,1) node{3}
   (9,1) node{5} 
   (2,2) node{21}
   (6,2) node{43}
   (3,3) node{321};
    \draw[nodes={draw, diamond, line width=.4pt, scale=0.7}, minimum height=1.9cm, minimum width=1.9cm,font=\footnotesize]
   (3,1) node{2}
   (4,2) node{32}
   (7,1) node{4};
   \draw[orange, line width=1.4pt] (-1,0)--(0,1)--(1,2)--(2,3)--(3,4)--(4,3)--(5,2)--(6,3)--(7,2)--(8,1)--(9,2)--(11,0);
   \draw[orange, line width=1.8pt] (-1,0)--(3,4)
    (3,0)--(6,3)
    (7,0)--(9,2);
\end{tikzpicture}
\]
\caption{A Dyck path with extended ascents}\label{Ascents}
\end{figure}

Now, for any block $T_i^j$ that appears on an extended ascent, we draw a subpath $P_D(i,j)$ according to the following instructions:
\begin{enumerate}
  \item Draw a path from the $x$-axis past blocks $T_i^i, T_i^{i+1},\hdots$ up to the peak of block $T_i^j$.
  \item \label{peak} From there, the path descends until hits another extended ascent, or returns to the $x$-axis.
  \item If the path hits an extended ascent, take one step up, and then go back to step (\ref{peak}).
  \item When the path returns to the $x$-axis, it is complete.
\end{enumerate}

\begin{ex}
If $D$ is the Dyck path in Figure \ref{Ascents}, then we obtain:

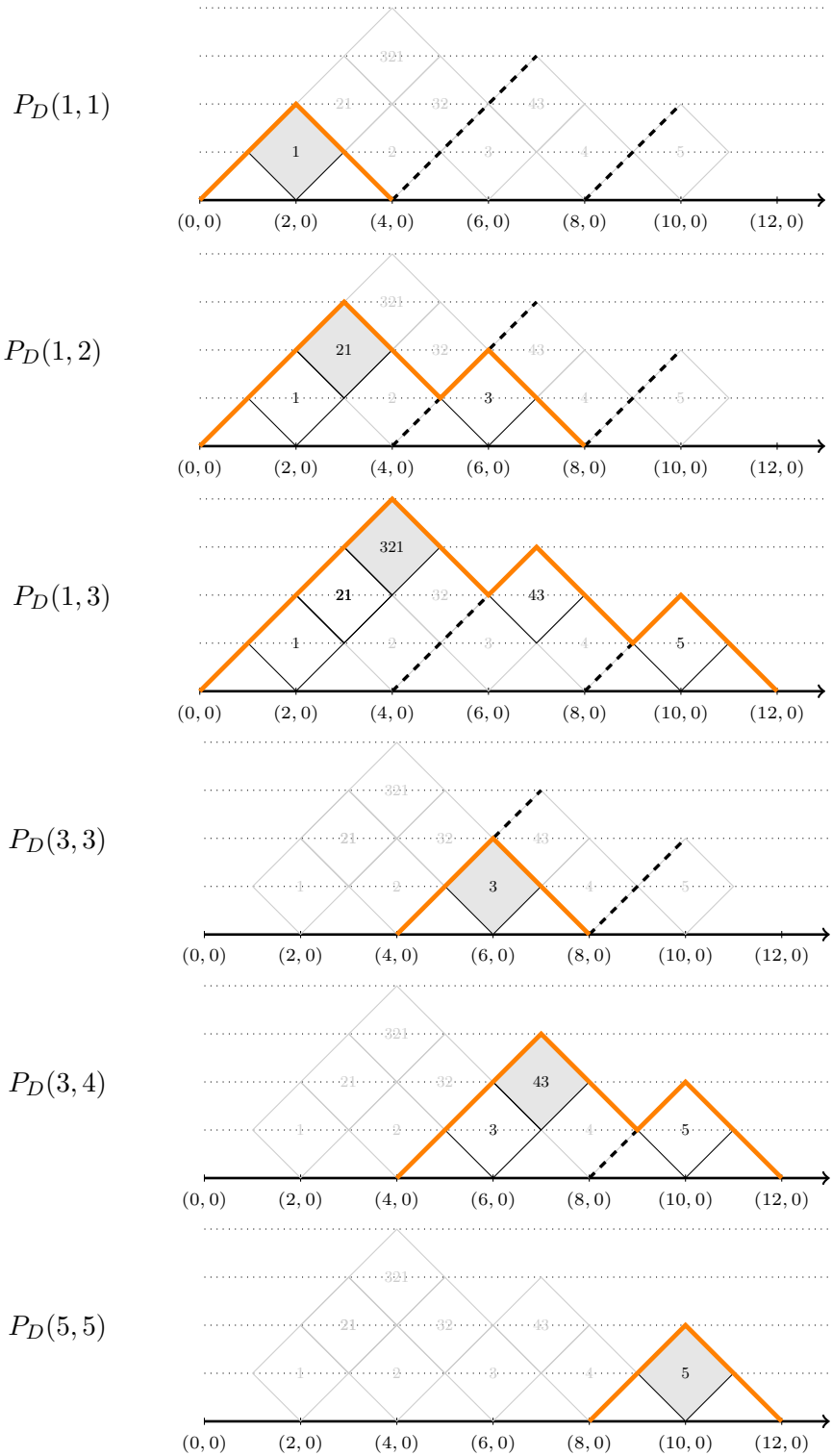
\[ 
P_D(1,1) \qquad  
\raisebox{-50pt}{\begin{tikzpicture}[x=.95cm,y=.95cm, scale=0.7]
\foreach \x in {0,2,...,12}
\draw[shift={(\x-1,0)},color=black, ] (0pt,2pt) -- (0pt,-2pt) node[below] {\tiny $(\x,0)$};
\draw[line width=1 pt, ->] (-1,0)--(12,0);
\foreach \y in {1,..., 4}
	\draw[dotted](-1,\y)--(12,\y);
     \draw[nodes={draw, diamond, line width=.4pt, scale=0.7, color=black!20}, minimum height=1.9cm, minimum width=1.9cm,font=\footnotesize]
     (3,1) node{2}
   (5,1) node{3}
   (9,1) node{5} 
   (2,2) node{21}
   (6,2) node{43}
   (3,3) node{321}
   (4,2) node{32}
   (7,1) node{4};
   \draw[nodes={draw, diamond, line width=.4pt, scale=0.7,fill=gray!20}, minimum height=1.9cm, minimum width=1.9cm,font=\footnotesize]	
       (1,1) node{1};
   \draw[dashed,line width=1.3pt] 
     (3,0)--(6,3)
     (7,0)--(9,2);
     \draw[orange, line width=1.8pt] (-1,0)--(0,1)--(1,2)--(2,1)--(3,0);
\end{tikzpicture}}
\]

\[
P_D(1,2) \qquad 
\raisebox{-50pt}{
\begin{tikzpicture}[x=.95cm,y=.95cm, scale=0.7]
\foreach \x in {0,2,...,12}
\draw[shift={(\x-1,0)},color=black, ] (0pt,2pt) -- (0pt,-2pt) node[below] {\tiny $(\x,0)$};
\draw[line width=1 pt, ->] (-1,0)--(12,0);
\foreach \y in {1,..., 4}
	\draw[dotted](-1,\y)--(12,\y);
\draw[nodes={draw, diamond, line width=.4pt, scale=0.7, color=black!20}, minimum height=1.9cm, minimum width=1.9cm,font=\footnotesize]
   (3,1) node{2}
   (4,2) node{32}
   (6,2) node{43}
   (9,1) node{5} 
   (3,3) node{321}
   (7,1) node{4};
     \draw[nodes={draw, diamond, line width=.4pt, scale=0.7}, minimum height=1.9cm, minimum width=1.9cm,font=\footnotesize]
   (1,1) node{1}
   (5,1) node{3}
   (2,2) node{21};
        \draw[nodes={draw, diamond, line width=.4pt, scale=0.7,fill=gray!20}, minimum height=1.9cm, minimum width=1.9cm,font=\footnotesize]	
      (2,2) node{21};
   \draw[dashed,line width=1.3pt] 
     (3,0)--(6,3)
     (7,0)--(9,2);
     \draw[orange, line width=1.8pt] (-1,0)--(0,1)--(1,2)--(2,3)--(3,2)--(4,1)--(5,2)--(6,1)--(7,0);
\end{tikzpicture} }
\]
\[
P_D(1,3) \qquad
\raisebox{-50 pt}{\begin{tikzpicture}[x=.95cm,y=.95cm, scale=0.7]
\foreach \x in {0,2,...,12}
\draw[shift={(\x-1,0)},color=black, ] (0pt,2pt) -- (0pt,-2pt) node[below] {\tiny $(\x,0)$};
\draw[line width=1 pt, ->] (-1,0)--(12,0);
\foreach \y in {1,..., 4}
	\draw[dotted](-1,\y)--(12,\y);
       \draw[nodes={draw, diamond, line width=.4pt, scale=0.7,color=black!20}, minimum height=1.9cm, minimum width=1.9cm,font=\footnotesize]
   (3,1) node{2}
   (5,1) node{3}
   (4,2) node{32}
   (7,1) node{4};
     \draw[nodes={draw, diamond, line width=.4pt, scale=0.7}, minimum height=1.9cm, minimum width=1.9cm,font=\footnotesize]
   (1,1) node{1}
   (2,2) node{21}
   (9,1) node{5} 
   (2,2) node{21}
   (6,2) node{43};
     \draw[nodes={draw, diamond, line width=.4pt, scale=0.7,fill=gray!20}, minimum height=1.9cm, minimum width=1.9cm,font=\footnotesize]	
   (3,3) node{321};
   \draw[dashed,line width=1.3pt] 
     (3,0)--(6,3)
     (7,0)--(9,2);
     \draw[orange, line width=1.8pt] (-1,0)--(0,1)--(1,2)--(2,3)--(3,4)--(4,3)--(5,2)--(6,3)--(7,2)--(8,1)--(9,2)--(10,1)--(11,0);
\end{tikzpicture}} \]
\[
P_D(3,3) \qquad \raisebox{-50pt}{
\begin{tikzpicture}[x=.95cm,y=.95cm, scale=0.7]
\foreach \x in {0,2,...,12}
\draw[shift={(\x-1,0)},color=black, ] (0pt,2pt) -- (0pt,-2pt) node[below] {\tiny $(\x,0)$};
\draw[line width=1 pt, ->] (-1,0)--(12,0);
\foreach \y in {1,..., 4}
	\draw[dotted](-1,\y)--(12,\y);
     \draw[nodes={draw, diamond, line width=.4pt, scale=0.7,fill=gray!20}, minimum height=1.9cm, minimum width=1.9cm,font=\footnotesize]	
   (5,1) node{3};
     \draw[nodes={draw, diamond, line width=.4pt, scale=0.7,color=black!20}, minimum height=1.9cm, minimum width=1.9cm,font=\footnotesize]
   (1,1) node{1}
   (2,2) node{21}
   (9,1) node{5} 
   (2,2) node{21}
   (6,2) node{43}
   (3,3) node{321}
   (3,1) node{2}
   (4,2) node{32}
   (7,1) node{4};
   \draw[dashed,line width=1.3pt] 
     (3,0)--(6,3)
     (7,0)--(9,2);
     \draw[orange, line width=1.8pt] (3,0)--(4,1)--(5,2)--(6,1)--(7,0);
\end{tikzpicture}}
\]
\[
P_D(3,4) \qquad \raisebox{-50pt}{
\begin{tikzpicture}[x=.95cm,y=.95cm, scale=0.7]
\foreach \x in {0,2,...,12}
\draw[shift={(\x-1,0)},color=black, ] (0pt,2pt) -- (0pt,-2pt) node[below] {\tiny $(\x,0)$};
\draw[line width=1 pt, ->] (-1,0)--(12,0);
\foreach \y in {1,..., 4}
	\draw[dotted](-1,\y)--(12,\y);
        \draw[nodes={draw, diamond, line width=.4pt, scale=0.7,color=black!20}, minimum height=1.9cm, minimum width=1.9cm,font=\footnotesize]
   (3,1) node{2}
   (1,1) node{1}
   (2,2) node{21}
   (4,2) node{32}
     (3,3) node{321}
   (7,1) node{4};
    \draw[nodes={draw, diamond, line width=.4pt, scale=0.7,fill=gray!20}, minimum height=1.9cm, minimum width=1.9cm,font=\footnotesize]	
   (6,2) node{43};
     \draw[nodes={draw, diamond, line width=.4pt, scale=0.7}, minimum height=1.9cm, minimum width=1.9cm,font=\footnotesize]
   (5,1) node{3}
   (9,1) node{5};

   \draw[dashed,line width=1.3pt] 
     (3,0)--(6,3)
     (7,0)--(9,2);
     \draw[orange, line width=1.8pt] (3,0)--(4,1)--(5,2)--(6,3)--(7,2)--(8,1)--(9,2)--(11,0);
\end{tikzpicture}}
\]
\[
P_D(5,5) \qquad \raisebox{-50 pt}{
\begin{tikzpicture}[x=.95cm,y=.95cm, scale=0.7]
\foreach \x in {0,2,...,12}
\draw[shift={(\x-1,0)},color=black, ] (0pt,2pt) -- (0pt,-2pt) node[below] {\tiny $(\x,0)$};
\draw[line width=1 pt, ->] (-1,0)--(12,0);
\foreach \y in {1,..., 4}
	\draw[dotted](-1,\y)--(12,\y);
     \draw[nodes={draw, diamond, line width=.4pt, scale=0.7,fill=gray!20}, minimum height=1.9cm, minimum width=1.9cm,font=\footnotesize]	
   (9,1) node{5};
\draw[nodes={draw, diamond, line width=.4pt, scale=0.7, color=black!20}, minimum height=1.9cm, minimum width=1.9cm,font=\footnotesize]
   (3,1) node{2}
   (4,2) node{32}
   (7,1) node{4}
      (1,1) node{1}
   (2,2) node{21}
   (5,1) node{3} 
   (2,2) node{21}
   (6,2) node{43}
   (3,3) node{321};
   \draw[dashed,line width=1.3pt] 
     (7,0)--(9,2);
     \draw[orange, line width=1.8pt] (7,0)--(8,1)--(9,2)--(11,0);
\end{tikzpicture}}
\]
\end{ex}

Now we define the number $p_D(i,j)$ by
\begin{equation} \label{eq-pd} p_D(i,j) :=  (\# \text{ of steps in the ascents of } P_D(i,j)) -1. \end{equation}
Then we observe
\begin{align*} p_D(i,j) & =  (\# \text{ of blocks in the ascents}) \\ &=
(\# \text{ of peaks})+(\text{height of the first peak}) - 2. \end{align*}
Recall that we have the map $\Psi$ from the set of Dyck paths into the set of fully commutative elements. Now we state the main result of this section:
\begin{prop}
\label{AForm}
Assume that a Dyck path $D$ does not have an ascent longer than $1$ step except for an ascent beginning on the $x$-axis. Then the dimension $d_D$ of the homogeneous module  $S(\Psi(D))$ is given by the formula
\begin{equation} \label{ff}
   d_D = \prod \frac{k!}{p_D(i,j)}
\end{equation}
where $k =(\textrm{sum of peak heights}) - (\textrm{\# of peaks})$ for the path $D$ and the product runs through all blocks $T_i^j$ on the extended ascents of $D$.
\end{prop}

A proof of the above proposition will be given in the rest of this section. Let us see an example before we begin the proof.

\begin{ex} \label{ex-dim}
  For the path $D$ in Figure \ref{Ascents}, the numbers $p_D(i,j)$ are shown here:
  \[
     \begin{tabular}{ r |c | c | c | c | c |c }
     $(i,j)$ & $(1,1)$ & $(1,2)$ & $(1,3)$ & $(3,3)$ & $(3,4)$ & $(5,5)$\\\hline
      $p_D(i,j)$ & 1 & 3 & 5 & 1 &  3 & 1
     \end{tabular}
  \]
  Then the dimension of the homogeneous module corresponding to the fully commutative element $\Psi(D)=321435$ is 
  \[
     d_D= \frac{6!}{1\cdot3\cdot5\cdot1\cdot3\cdot1} = 16.
  \]
\end{ex}

\medskip

Recall that an element $\bw \in W$ is called {\em dominant minuscule} if there is a dominant integral weight $\Lambda$ and a reduced expression $\bw =s_{i_1} s_{i_2} \cdots s_{i_d}$ such that 
\[   s_{i_k} s_{i_{k+1}} \cdots s_{i_d} \Lambda = \Lambda -\alpha_{i_k} -\alpha_{i_{k+1}} - \cdots - \alpha_{i_d} \qquad (1 \le  k \le d) .\]
It is known that dominant minuscule elements are fully commutative. We have the following characterization of dominant minuscule elements.

\begin{prop} \cite[Proposition 2.5]{Stembridge2001} \label{dm}
If $\bw = s_{i_1} s_{i_2} \cdots s_{i_d} \in W$ is a reduced expression, then $\bw $  is dominant minuscule if and only if the following two
conditions are satisfied:
\begin{enumerate}
\item between every pair of occurrences of a generator $s_i$ (with no other
occurrences of $s_i$ in between) there are exactly two generators (possibly
equal to each other) that do not commute with $s_i$;
\item the last occurrence of each generator $s_i$ is followed by at most one
generator that does not commute with $s_i$.

\end{enumerate}

\end{prop}

 For a Dyck path $D$, it is clear that $\Psi(D)^{-1}$ is also a fully commutative element. The following corollary characterizes dominant minuscule elements using shapes of Dyck paths. One can compare it with Lemma 3.9 in \cite{Klesh2008}, where {\em straight shapes} are used.

\begin{lemma} \label{cor-ascent}
Let $D$ be a Dyck path. Then $\Psi(D)^{-1}$ is dominant minuscule if and only if  any ascent in $D$ not beginning on the $x$-axis has a length of 1. 
  \end{lemma}

\begin{proof} 
Assume that  $D$ has no ascents longer than $1$ besides those that begin on the $x$-axis.
Since we must have a descent and then an ascent to get from one peak to the next, we see that the condition (i) of Proposition \ref{dm} is satisfied by $\Psi(D)$ and $\Psi(D)^{-1}$. Write $\Psi(D)=T_{i_1}^{m_{1}}T_{i_2}^{m_{2}}\cdots T_{i_\ell}^{m_{\ell}}$ as before. Then every generator in $T_{i_1}^{m_1}$ first appears in $\Psi(D)$ and there is at most one generator before its occurrence that does not commute with it. Thus $\Psi(D)^{-1}$ satisfies the condition (ii) of Proposition \ref{dm} with the generators in $T_{i_1}^{m_1}$. 

Consider now the peak corresponding to the segment $T^{m_2}_{i_2}$.  If we arrive there after an ascent of length 1, then $m_2=m_1+1$ and $i_1<i_2$. Thus the only new generator appearing in $T_{i_2}^{m_2}$ is $s_{m_2}=s_{m_1+1}$ and it commutes with all the generators preceding it except $s_{m_1}$.  On the other hand, if we arrive at this peak after following an ascent longer than 1 step then, by assumption, this ascent begins on the $x$-axis.  Then, we necessarily find that $i_2 > m_1 + 1$.  So every generator in $T_{i_2}^{m_2}$ appears here for the first time, but commutes with all generators appearing previously.  We can continue inductively, analyzing the generators appearing for the first time in each segment $T_{i_j}^{m_j}$,  and see that the condition (ii) of Proposition \ref{dm} is satisfied by $\Psi(D)^{-1}$. Therefore, the element $\Psi(D)^{-1}$ is dominant minuscule. 

Conversely, if $\Psi(D)^{-1}$ is dominant minuscule, the condition (ii) of Proposition \ref{dm} implies that any generator appearing for the first time in a segment $T_{i_j}^{m_j}$ will either commute with all previously appearing generators (thus the ascent corresponding peak begins on the $x$-axis), or that it does not commute with exactly one previously appearing generator (thus $m_{j-1} = m_j -1$, and the ascent was of length 1). 
  \end{proof}

\begin{proof} [Proof of Proposition \ref{AForm}]
We will obtain the formula \eqref{ff} as a reformulation of the Peterson-Proctor formula \cite[Theorem 3.10]{Klesh2008}. Write $\bw=\Psi(D)$. It follows from Lemma \ref{cor-ascent} that $\bw^{-1}$ is dominant minuscule. Then we only need to establish two things: First, a bijective correspondence between $\{ \beta \in \Delta^+ \,|\, \bw (\beta) < 0 \}$ and $\{ P_D(i,j) \,|\, T_i^j \text{ is on the extended ascents of } D \}$, where $\Delta^+$ is the set of positive roots. Second, the equality $\mathrm{ht}(\beta)=p_D(i,j)$ when $\beta$ corresponds to the path $P_D(i,j)$.

We write $\Psi(D)=T_{i_1}^{m_{1}}T_{i_2}^{m_{2}}\cdots T_{i_\ell}^{m_{\ell}}$. Each $\beta \in \Delta^+$ with $\bw(\beta)<0$ determines a unique $(i_k,n_k)$, $i_k \le n_k \le m_k$, such that
\[ \beta=\alpha_{n_k}T_{i_k}^{n_k-1}T_{i_{k+1}}^{m_{k+1}} \cdots T_{i_l}^{m_l} = \alpha_{n_k}T_{i_k}^{n_k-1}T_{i_{k+1}}^{n_{k}+1}T_{i_{k+2}}^{n_{k}+2} \cdots T_{i_l}^{n_k+l-k}, \] where the action on $\alpha_{n_k}$ is from the right. On the other hand, each block $T_{i_k}^{n_k}$, $i_k \le n_k \le m_k$, is on an extended ascent and 
\[ \Psi(P_D(i_k, n_k))= T_{i_k}^{n_k}T_{i_{k+1}}^{n_{k}+1}T_{i_{k+2}}^{n_{k}+2} \cdots T_{i_l}^{n_k+l-k}. \] Then the correspondence $\beta \mapsto P_D(i_k,n_k)$ is clearly one-to-one and onto.  

Furthermore, we see that
\[ \beta = \alpha_{n_k}T_{i_k}^{n_k-1}T_{i_{k+1}}^{n_{k}+1}T_{i_{k+2}}^{n_{k}+2} \cdots T_{i_l}^{n_k+l-k} = (\alpha_{i_k}+ \cdots + \alpha_{n_k})+ \alpha_{n_k+1}+ \cdots + \alpha_{n_k+l-k}, \] and $\mathrm {ht}(\beta)= n_k-i_k+1+l-k=(\# \text{ of steps in the ascents}) -1=p_D(i_k,n_k)$ from \eqref{eq-pd}. This completes the proof.

\end{proof}

Even when Proposition \ref{AForm} does not apply directly, we may still find the dimension of the corresponding module: We can
  \begin{itemize}
   \item Consider the reverse path (reflected left to right), or  
   \item Invert the corresponding fully commutative element, and consider the associated Dyck path.
  \end{itemize}
Note that reversing a path corresponds to the graph automorphism of the Dynkin diagram.
  The two options would give distinct paths, but if either satisfies the condition of Proposition \ref{AForm}, then we can  obtain the correct dimension using the formula.

  \begin{ex}
  The path $D$ in Figure \ref{nonex} below does not satisfy the condition of Proposition \ref{AForm}.  
  However, we note that the reverse of the path $D$  is nothing but the path in Figure \ref{Ascents}, for which we computed the dimension in Example \ref{ex-dim}. Thus we obtain the same dimension, $16$, for the homogeneous representation corresponding to $D$.
     \begin{figure}[h]
\[
\begin{tikzpicture}[x=.95cm,y=.95cm, scale=0.7]
\foreach \x in {0,-2,...,-10}
\draw[shift={(\x-1,0)},color=black, ] (0pt,2pt) -- (0pt,-2pt);
\draw[line width=1 pt, <-] (2,0)--(-12,0);
\foreach \y in {1,..., 4}
  \draw[dotted](1,\y)--(-12,\y);
   \draw[orange, line width=1.8pt] (1,0)--(-3,4)--(-5,2)--(-6,3)--(-8,1)--(-9,2)--(-11,0);
\end{tikzpicture}
\]
\caption{A Dyck Path for which the formula does not work directly}\label{nonex}
\end{figure}
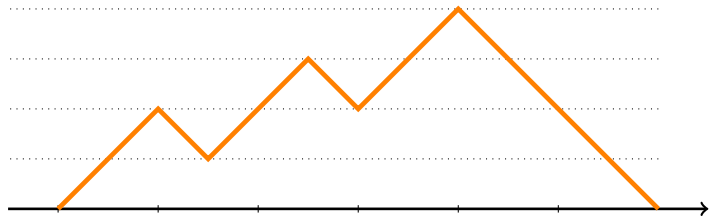

\end{ex}

\vskip 1 cm

\bibliographystyle{amsplain}

\begin{thebibliography}{10}

\bibitem{Benk2011}
Georgia Benkart, Seok-Jin Kang, Se-jin Oh, and Euiyong Park, \emph{Construction
  of irreducible representations over khovanov-lauda-rouquier algebras of
  finite classical type}, Int. Math. Res. Notices \textbf{2014} (2014), no.~5,
  1312--1366.

\bibitem{Bokut2001}
Leonid~A. Bokut and Long-Sheng Shiao, \emph{Gr\"obner-{S}hirshov bases for
  {C}oxeter groups}, Comm. Algebra \textbf{29} (2001), no.~9, 4305--4319,
  Special issue dedicated to Alexei Ivanovich Kostrikin.

\bibitem{Brundan2009}
Jonathan Brundan and Alexander Kleshchev, \emph{{Blocks of cyclotomic Hecke
  algebras and Khovanov-Lauda algebras}}, Invent. Math. \textbf{178} (2009),
  no.~3, 451--484.

\bibitem{Deutsch1999}
Emeric Deutsch, \emph{{Dyck path enumeration}}, Discrete Math. \textbf{204}
  (1999), no.~1-3, 167--202.

\bibitem{Tnk}
\bysame, \emph{Sequence {A}140717}, \href{http://oeis.org/A140717}{The
  {O}n-{L}ine {E}ncyclopedia of {I}nteger {S}equences}, 2008, published
  electronically at http://oeis.org.

\bibitem{Fan1996}
C.~Kenneth Fan, \emph{{A Hecke Algebra quotient and some combinatorial
  applications}}, J. Algebraic Combin. \textbf{5} (1996), no.~3, 175--189.

\bibitem{Hill2012}
David Hill, George Melvin, and Damien Mondragon, \emph{{Representations of
  quiver Hecke algebras via Lyndon bases}}, J. Pure Appl. Algebra \textbf{216}
  (2012), no.~5, 1052--1079.

\bibitem{Jones1987}
Vaughan F.~R. Jones, \emph{{Hecke algebra representations of braid groups and
  link polynomials}}, Ann. of Math. (2) \textbf{126} (1987), no.~2, 335--388.

\bibitem{Kang2011}
Seok-Jin Kang and Masaki Kashiwara, \emph{{Categorification of highest weight
  modules via Khovanov-Lauda-Rouquier algebras}}, Invent. Math. \textbf{190}
  (2012), no.~3, 699--742.

\bibitem{Khov2009}
Mikhail Khovanov and Aaron~D. Lauda, \emph{{A diagrammatic approach to
  categorification of quantum groups I}}, Represent. Theory \textbf{13} (2009),
  no.~09, 309--347.

\bibitem{Klesh2008}
Alexander Kleshchev and Arun Ram, \emph{{Homogeneous representations of
  Khovanov-â€“Lauda Algebras}}, J. Eur. Math. Soc. \textbf{12} (2010),
  no.~5, 1293--1306.

\bibitem{Klesh2010}
\bysame, \emph{{Representations of Khovanov-â€“Laudaâ€“-Rouquier
  algebras and combinatorics of Lyndon words}}, Math. Ann. \textbf{349} (2011),
  no.~4, 943--975.

\bibitem{Lauda2011}
Aaron~D. Lauda and Monica Vazirani, \emph{{Crystals from categorified quantum
  groups}}, Adv. Math. \textbf{228} (2011), no.~2, 803--861.

\bibitem{lusztig2011introduction}
George Lusztig, \emph{Introduction to quantum groups}, Springer, 2011.

\bibitem{Rouq2008}
Rapha\"el Rouquier, \emph{2-kac-moody algebras},  (2008), {\tt arXiv:0812.5023
  [math.RT]}.

\bibitem{Stembridge1996}
John~R. Stembridge, \emph{On the fully commutative elements of coxeter groups},
  J. Algebraic Combin. \textbf{7} (1996), 353--385.

\bibitem{Stembridge1998}
\bysame, \emph{{The enumeration of fully commutative elements of Coxeter
  groups}}, J. Algebraic Combin. \textbf{7} (1998), no.~3, 291--320.

\bibitem{Stembridge2001}
\bysame, \emph{Minuscule elements of weyl groups}, J. Algebra (2001), 722--743.

\end{thebibliography}

\end{document}